\documentclass[13pt]{scrartcl}

\usepackage{libertinus}   % fonte elegante
\usepackage{microtype}

\usepackage[utf8]{inputenc}
\RequirePackage[l2tabu, orthodox]{nag}
\usepackage[T1]{fontenc}
\usepackage{lmodern}
\usepackage{amsthm,amssymb,amsmath}
\usepackage{dsfont}
\usepackage{latexsym}
\usepackage{graphicx}
\usepackage{subfigure}
\usepackage{tikz}
\usetikzlibrary{arrows}
\usetikzlibrary{calc}
\usepackage[linktocpage=true]{hyperref}
\usepackage[left=2.5cm, right=2.5cm, top=2.5cm, bottom=2.5cm]{geometry}
\usepackage{color}
\usepackage{here}
\usepackage{mathrsfs}
\usepackage{pgf,tikz}
\usetikzlibrary{decorations.pathreplacing, shapes.multipart, arrows, matrix, shapes}
\usetikzlibrary{patterns}
\usepackage{caption}
\usepackage[english]{babel}
\usepackage[all] {xy}
\usepackage{enumitem}
\usepackage{cancel}
\usepackage{comment}
\usepackage{verbatim}

\makeatletter
\newcommand\footnoteref[1]{\protected@xdef\@thefnmark{\ref{#1}}\@footnotemark}
\makeatother

\hypersetup{ colorlinks   = true, %Colours links instead of ugly boxes
	urlcolor  = blue, %Colour for external hyperlinks
	linkcolor    = blue, %Colour of internal links
	citecolor   = blue %Colour of citations
}

\date{}

\def\beq{\begin{equation}}
	\def\eeq{\end{equation}}

\def\det{\textrm{det}\ }

\newcommand\restr[2]{{ %comando pra restrição de função ficar bonita
		\left.\kern-\nulldelimiterspace 
		#1 
		\vphantom{\big|} 
		\right|_{#2} 
}}

\newcommand{\Z}{{\mathbb Z}}
\newcommand{\R}{{\mathbb R}}

\newcommand{\N}{{\mathbb N}}

\newcommand{\TT}{{\mathbb{T}^3}}

\newcommand{\cB}{{\mathcal B}}

\newcommand{\cN}{{\mathcal N}}
\newcommand{\cP}{{\mathcal P}}

\newcommand{\cW}{{\mathcal W}}

\newcommand{\eps}{\varepsilon}
\newcommand{\per}{\operatorname{Per}(f)}
\newcommand{\card}{\operatorname{card}}
\newcommand{\cp}{\mathcal{P}^{\operatorname{erg}}_f(M)}

\newcommand{\eqdef}{\stackrel{\scriptscriptstyle\textrm def.}{=}}

\newcommand{\Leb}{\operatorname{Leb}}

\newcommand{\jac}{\operatorname{J}}
\makeatletter
\newcommand{\tpitchfork}{%
	\vbox{
		\baselineskip\z@skip
		\lineskip-.52ex
		\lineskiplimit\maxdimen
		\m@th
		\ialign{##\crcr\hidewidth\smash{$-$}\hidewidth\crcr$\pitchfork$\crcr}
	}%
}
\makeatother

\newtheorem{theorem}{Theorem}[section]
\newtheorem{corollary}{Corollary}
\newtheorem{maintheorem}{Theorem}

\newtheorem{lemma}[theorem]{Lemma}
\newtheorem{proposition}{Proposition}

\theoremstyle{definition}
\newtheorem{definition}[theorem]{Definition}
\newtheorem{remark}{Remark}

\title{Rigidity for $u$-Gibbs measures admitting Margulis systems}
\author{Vítor Gomes and Bruno Santiago}
\date{\today}

\begin{document}

	\maketitle
	
	\begin{abstract}
		Measures admitting Margulis systems and $u$-Gibbs measures arise from different geometric principles in partially hyperbolic dynamics. In this paper we investigate the implications of the coincidence of these two structures. We obtain new entropy rigidity results for $u$-Gibbs measures by showing that whenever a $u$-Gibbs measure of a partially hyperbolic diffeomorphism admits an unstable Margulis family, the unstable Jacobian data of the system must to be constant. We apply our result to center isometries and flow type diffeomorphisms showing that if a measure of maximal entropy is also $u$-Gibbs then Jacobian periodic data along the unstable bundle are constant.  In the case of smooth jointly integrable partially hyperbolic diffeomorphisms of $\TT$, assuming that there exists some $u$-Gibbs measure which is also a measure of maximal unstable entropy, we obtain smooth conjugacy along the center-unstable foliation and uniqueness of $u$-Gibbs measures in this case. 
	\end{abstract}	
	
	\tableofcontents
	
	\section{Introduction}
	
	Entropy is a fundamental invariant in dynamics, which gives the asymptotic growth rate of the number of distinguishable orbits, up to some arbitrarily small scale. A generalization of entropy is the concept of pressure, which is a weighted version: instead of simply counting one weights each distinguishable orbit with the exponential of the Birkhoff sum of a given potential. In this sense, entropy is nothing but the pressure with respect to the constant zero potential. The notion of pressure - and in particular entropy - can be given in the topological and in the measure theoretical sense. A beautiful result in basic ergodic theory is the variational principle of Walters \cite{W}, which asserts that the topological pressure is the supremum over all invariant probability measures of the sum of the entropy with the average value of the potential. 
	
	The variational principle hints that measures achieving the supremum must to be dynamically relevant. We call such measures equilibrium states. Naturally enough, the quest for a satisfactory theory of equilibrium states is a primary goal in ergodic theory, still far beyond reach for several important classes of dynamical systems.   
	
	However, in the uniformly hyperbolic setting we have a fairly complete theory of equilibrium states, nowadays given in book-from, for instance in Chapter 20 of \cite{KH} or the classical monograph \cite{BowenR} of Bowen. We say that a diffeomorphism $f:M\to M$ of a smooth $d$-dimensional manifold $M$ is an Anosov diffeomorphism if the tangent bundle decomposes $TM=E^s\oplus E^u$ into invariant sub-bundles, with respect to the deferential $Df$ of $f$, so that vectors in the \emph{stable bundle} $E^s$ are exponentially contracted under future iteration while vectors in the \emph{unstable bundle} $E^u$ are exponentially contracted under past iteration. These bundles uniquely integrate into invariant foliations, with smooth leaves (as smooth as the map) with continuous foliated charts.

	We know that given an Anosov diffeomorphism with a dense orbit (transitive, which conjecturally are all of them), to each Hölder continuous potential we can associate a unique equilibrium state. This is a source for the construction of interesting ergodic invariant measures. For instance, if we take as potential the Jacobian determinant along the unstable bundle, the equilibrium state is a measure absolutely continuous along the unstable foliation, the so-called SRB measure.  
	
	Moreover, the theory also predicts that two equilibrium states coincide if, and only if, their respective potentials satisfy a co-homological equation - in particular, they must have the same Birkhoff sum at every periodic orbit. A particularly beautiful corollary of this theory is the follwoing: if we deform a linear Anosov torus automorphism but retaining the property that the entropy of the unique SRB measure equals the topological entropy then the Jacobian data along the unstable bundle must to be constant over all periodic orbits. For two dimensional systems, one may deduce from this the existence of a conjugacy with a linear model, smooth along the unstable foliation, thus getting a strong form of rigidity. In the conservative case, one can get even full smoothness of such conjugacy, see Theorem 20.21 of \cite{KH}, or the original reference \cite{LlaveMarcoMoriyon1987} for smooth conjugacy from coincidence of periodic data.
	 
	This is perhaps the best known example we have of an entropy rigidity statement for measures. The main motivation for the present paper is to find new examples of similar entropy rigidity statements. 
	
	In this direction, one of the most natural generalizations of uniform hyperbolicity is the class of \emph{partially hyperbolic systems}. We say that a diffeomorphism $f:M\to M$ is partially hyperbolic if there exists a decomposition of the tangent bundle $TM=E^s\oplus E^c\oplus E^u$ which is $Df$-invariant, such that vectors in the stable bundle $E^s$ are contracted, vectors in $E^u$ are expanded (contracted exponentially under past iteration) while vectors in the \emph{center bundle} $E^c$ have a dominated behaviour: they are never as contracted as any vector in $E^s$ neither they are as expanded as vectors in $E^u$ (see Section~\ref{sec.dois} for a more complete treatment). The stable and the unstable bundle uniquely integrate into invariant foliations $\cW^s$ and $\cW^u$ with smooth leaves depending continuously of the base point. 
	
	For partially hyperbolic systems we are still very far from a complete theory of equilibrium states, although this is a rich and active topic of research in the field, see the introduction of \cite{CPZ} and the references therein for some account. This is the main reason why we search here for entropy rigidity of \emph{$u$-Gibbs measures}. An invariant measure under a partially hyperbolic map $f:M\to M$ is said to be $u$-Gibbs if conditionals along the unstable foliation are absolutely continuous with respect to the inner Lebesgue measure on the leafs (defined with the inherited Riemannian structure from the ambient $M$). This is a natural replacement for SRB measures in this case: $u$-Gibbs measures always exist in this setting \cite{PS}. However, in general $u$-Gibbs measures are not equilibrium states, even for systems which are simultaneously Anosov and partially hyperbolic. 
	
	We take this gap in the theory as a motivation, for in this work we adopt a different approach towards entropy rigidity. We focus on the geometrical properties of conditional measures instead. Moreover, we consider systems of locally finite Borel measures defined along entire leaves of the unstable foliation which in restriction to measurable subordinated partitions agree with the conditional measures almost surely and up to normalization. The geometry of the measure is characterized dynamically by the way the measure changes when we move from $x$ to $f(x)$. 
	
	For instance, we say that an invariant and ergodic measure $\mu$ admits \emph{a Margulis system of unstable measures} if there exists a family $\{m^u_x\}_{x\in M}$ of locally finite Borel measures along the unstable leaf through $x$, which are uniquely defined along the leaf, coincides with conditional measures on foliated boxes, up to normalization, and satisfy
	\[
	m^u_x\circ f=e^{\lambda}m^u_{f^{-1}(x)},
	\]  
	for some constant $\lambda>0$. 
	
	Although both classes of measures have become central objects in partially hyperbolic dynamics, they arise from completely different mechanisms. The $u$-Gibbs measures are defined through absolute continuity properties of unstable conditional measures, whereas Margulis systems originate from entropy-maximizing constructions. At present, very little is understood about when these two structures can coexist.

  Our main theorem shows that these two geometric structures are compatible only under strong rigidity of the unstable Jacobian cocycle.
    
    \begin{maintheorem}
    	\label{main.technical}
    	Let $f:M\to M$ be a partially hyperbolic diffeomorphism with minimal unstable foliation. Assume that there exists some ergodic $u$-Gibbs measure $\mu\in\cp$ which admits a Margulis system of unstable measures. Then, $f$ has constant unstable Jacobian periodic data. 	
    \end{maintheorem} 
    
   Since our rigidity argument requires comparing Rokhlin conditional measures with Margulis systems defined on complete unstable leaves, we first develop a canonical realization of conditional measures as locally finite measures supported on entire unstable leaves. Although technically inspired by previous stabilization procedures used in different settings, the object constructed here is different and adapted to arbitrary partially hyperbolic diffeomorphisms. This framework is the main technical tool used throughout the proofs.
    
%	To prove this result we take inspiration in a construction of \cite{ALOS}, the so-called leafwise quotient measure. Using this idea, we construct leafwise conditional measures along the unstable foliation for \emph{every} ergodic invariant measure. This means a family $\{\mu^u_x\}_{x\in M}$ of locally finite Radon measures defined almost everywhere on the \emph{entire} unstable leaf through $x$, which coincides with the disintegration on every measurable partition subordinated to the unstable foliation. This is a quite useful tool which bypass the difficulties caused by the non-measurability of the unstable foliation in general. We believe that this construction has an independent interest and for this reason we give a detailed argument, even though the ideas are quite similar to those used in \cite{ALOS}. 

	Moreover, during the proof  we are were also able to solve the following problem: suppose that a partially hyperbolic diffeomorphism with minimal unstable foliation admits some $u$-Gibbs measure whose conditionals have a density bounded on the entire leaf. Then, the unstable Jacobian data at periodic orbits must to be constant (see Theorem~\ref{teo.densiteborne} for the precise statement). To the best of our knowledge, this result is new even in the two dimensional Anosov case. It also generalises \cite{MicenaTahzibi2016} and Théorème 2.2.1 from \cite{alvaro}, by Alvarez and Bonatti, which is the same statement for the geodesic flow on negative curvature.

	\subsection{Some aplications of Theorem~\ref{main.technical}}

	Systems of unstable measures always exist for the measure of maximal entropy of a transitive Anosov diffeomorphism by the seminal work of Margulis. This has been an important source of inspiration for new constructions of measures of maximal entropy in the partially hyperbolic setting. Many recent works have succeeded in constructing Margulis-like systems of unstable measures in different settings. Theorem A provides a common mechanism behind several recent constructions of Margulis systems.
	
	 In the case of \emph{center isometries} (see Section~\ref{sec.dois} for the precise definition), the work of Carrasco and Rodriguez-Hertz \cite{CRH} shows the existence of a Margulis system. Using their result, we obtain
	
	\begin{maintheorem}
		\label{main.teob}
		Let $f:M\to M$ be a center isometry of class $C^2$,  with minimal stable and unstable foliations. Assume that $f$ admits a $u$-Gibbs measure which is a measure of maximal entropy. Then, $f$ has constant unstable periodic data. 	
	\end{maintheorem}	

	For partially hyperbolic diffeomorphisms of \emph{flow type} (See Definition 3.1 of \cite{BFT}), the existence of a Margulis systems for certain measures of maximal entropy is ensured by Theorem 3.6 of \cite{BFT}, by Buzzi, Fisher and Tahzibi. In particular, using our main result we also obtain constant unstable periodic Jacobian data, whenever a flow type partially hyperbolic diffeomorphism admits a measure of maximal entropy with non-positive center Lyapunov exponent which is also $u$-Gibbs.   
	
	Another important work to mention is the recent preprint by Humbert \cite{humbert} which shows the existence of a Margulis system  for partially hyperbolic Anosov systems with expanding center in dimension three (see Section~\ref{sec.dois} for precise definitions), for measures of maximal \emph{unstable entropy} (see Section~\ref{sec.dois} for the definition of unstable entropy). We thus get:

	\begin{maintheorem}
		\label{teo.main}
		Let $f:\TT\to\TT$ be a $C^{\infty}$ Anosov partially hyperbolic diffeomorphism with expanding center. Assume that there exists an ergodic $u$-Gibbs state which is also a measure of maximal unstable entropy. Then, $f$ has constant unstable periodic data.
	\end{maintheorem}	
	
	It is still an open line of investigation what further rigidity one may or may not obtain in the setting of Theorem~\ref{teo.main}, given the conclusion of only constant unstable periodic data. Nevertheless, if we assume aditionally that the system is \emph{jointly integrable}, we obtain a very strong form of rigidity, as the following result demonstrates.

	\begin{corollary}
		\label{cor.casoji}
		Let $f:\TT\to\TT$ be a $C^{\infty}$ Anosov partially hyperbolic diffeomorphism with expanding center. If the bundle $E^s\oplus E^u$ is integrable and $f$ has a Gibbs $u$-state with maximal unstable entropy then the conjugacy with the linearization is $C^{1+\alpha}$ along center-unstable manifolds. In particular, $f$ displays a unique $u$-Gibbs measure. 
	\end{corollary}
	
	\subsubsection*{Structure of the article}
	
	This paper is organized as follows: in Section~\ref{sec.dois} we give preliminary material on the ergodic theory of  partially hyperbolic dynamical systems which is important for this text. We finish the section deducing all of our results from Theorem~\ref{main.technical}. In Section~\ref{sec.leafwise} we present a technical construction of conditional measures defined on entire unstable leaves and study properties of the densities of a $u$-Gibbs measure within this framework. The construction developed in Section~\ref{sec.leafwise} provides the framework that allows us to compare Rokhlin conditional measures with Margulis systems on complete unstable leaves. In Section~\ref{sec.formula} we place ourselfs under the assumptions of Theorem~\ref{main.technical} and prove a key formula for the unstable Jacobian. The proof of Theorem~\ref{main.technical} is completed in Section~\ref{sec.final}, where we also give a new result, namely Theorem~\ref{teo.densiteborne}.
	
	\subsubsection*{Acknowledgments} 
	We thank Sébastien Alvarez, Odylo Costa, Ali Tahzibi and Davi Obata for useful discussions.{\small
	
	V.G. was supported by \emph{Coordenação de Aperfeiçoamento de Pessoal de Nível Superior (CAPES)}. B.S. was supported by \emph{Conselho Nacional de Densenvolvimento Cinetífico e Tecnológico (CNPQ)} via the grant \emph{Bolsa PQ 307994/2025-2} and by \emph{Fundação de Amparao a Pesquisa do Estado do Rio de Janeiro (FAPERJ)} via the grants \emph{JCNE E-26/204.571/2024} and \emph{JPF E-26/210.344/2022}. B.S. and V.G. were partially supported by \emph{Coordenação de Aperfeiçoamento de Pessoal de Nível Superior (CAPES) - Finance code 001}.
	}
	
	\small{We made no mathematical use of AI in the preparation of this manuscript. We have used chatGPT Go for spell check and English polishing of some sentences.}

	\section{Partially hyperbolic dynamics}\label{sec.dois}
	
	Let $M$ be a smooth $d$ dimensional manifold endowed with some Riemannian metric. We say that a $C^r$, for $r\geq 1$, diffeomorphism $f:M\to M$ is \emph{partially hyperbolic} if there exists a decomposition of the tangent bundle $TM=E^s\oplus E^c\oplus E^u$ into non-trivial $Df$-invariant sub-bundles satisfying the following assumptions
	\begin{enumerate}
		\item $\|Df(x)|_{E^s}\|<m(Df(x)|_{E^c})\leq\|Df(x)|_{E^c}\|<m(Df(x)|_{E^u})$ and
		\item $\|Df(x)|_{E^s}\|<1<m(Df(x)|_{E^u})$, for all $x\in M$.
	\end{enumerate}  
	In the above definition $m(.)$ stands for the co-norm of the linear map, i.e. if $L$ is a linear map between finite dimensional vector spaces, $m(L)=\|L^{-1}\|^{-1}$. The bundles $E^s$, $E^c$ and $E^u$ are called, respectively, the \emph{stable, center} and \emph{unstable} bundle.  
	
	In this paper we shall use the following simplifying notation for Jacobian determinants: 
	\[
	\jac^{\star}_x\eqdef\left|\det\|Df(x)|_{E^{\star}}\|\right|,\:\:\:\star=s,c,u.
	\]
	We also define, for an integer $k>0$, 
	\[
	\jac^{\star}_x(k)\eqdef\prod_{\ell=0}^{k-1}\jac^{\star}_{f^{\ell}(x)}.
	\]
	For $k<0$ we define $\jac^{\star}_x(k)\eqdef(\jac^{\star}_{f^k(x)}(-k))^{-1}$. Observe that if the bundle $E^{\star}$ is one-dimensional we have that $\jac^{\star}_x(k)=\|Df^k(x)|_{E^{\star}}\|$. This notation allows us to see more transparently derivatives along sub-bundles as multiplicative cocycles. 
	
	We can now give the definition of the type of dynamical system we treat in Theorem~\ref{teo.main}.
	
	\begin{definition}
		Let $f:M\to M$ be a partially hyperbolic diffeomorphism. When $m(Df(x)|_{E^c})>1$, for every $x\in M$, we say that $f$ is an \emph{Anosov partially hyperbolic diffeomorphism with expanding center}. 	
	\end{definition}
	
	Notice that if the bundle $E^c$ is one-dimensional the above definition is equivalent to requiring that $\jac^c_x>1$ for every $x\in M$.
	
	In the sequel we define the dynamical systems we consider in Theorem~\ref{main.teob}.
	
	\begin{definition}
		Let $f:M\to M$ be a partially hyperbolic diffeomorphism. We say that $f$ is a \emph{center isometry} if, for every $x\in M$ and for every unit vector $v\in E^c(x)$, one has $\|Df(x)v\|=1$. 
	\end{definition}
	
	We recall also that the bundles $E^s$ and $E^u$ uniquely integrate into $f$-invariant continuous foliations $\cW^s$ and $\cW^u$ respectively. The leaves of these foliations are $C^r$ immersed submanifolds. We denote by $\cW^{\star}_r(x)$ the disk of radius $r>0$ around $x$ in the inner topology of the leaf $\cW^{\star}$, inherited from the ambient Riemannian structure. The (unique) integrability of the center bundle $E^c$ is a more delicate issue, being false in some cases, and true in some other cases. We stress that, although this is not necessary to our proofs, for three dimensional partially hypernbolic Anosov diffeomorphisms with expanding center as well as for center isometries, the center bundle is uniquely integrable into an $f$ invariant foliation $\cW^c$.   
	
	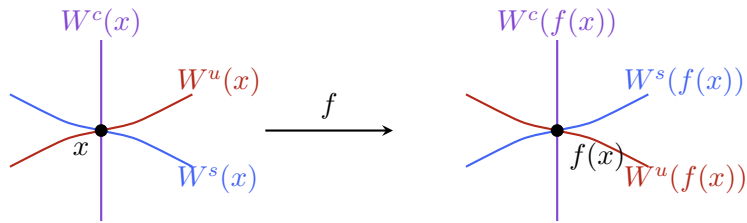
\begin{figure}[h]
		\centering
		\begin{tikzpicture}[scale=1.2, every node/.style={font=\footnotesize}]
			
			% cores das folheações
			\definecolor{stable}{RGB}{66, 100, 245}
			\definecolor{unstable}{RGB}{180, 40, 25}
			\definecolor{center}{RGB}{130,70,210}
			
			% --- Ponto x ---

			% --- Folheações em x ---
			\draw[stable, thick]
			plot [smooth] coordinates {(-1,0.4) (-0.4,0.1) (0,0) (0.4,-0.1) (1,-0.4)};
			\draw[unstable, thick]
			plot [smooth] coordinates {(-1,-0.4) (-0.4,-0.1) (0,0) (0.4,0.1) (1,0.4)};
			\draw[center, thick]
			plot [smooth] coordinates {(0,-1) (0,-0.3) (0,0) (0,0.3) (0,1)};
			
			% --- Rótulos ---
			\node[stable, below right] at (0.7,-0.3) {$W^s(x)$};
			\node[unstable, above right] at (0.7,0.25) {$W^u(x)$};
			\node[center, above] at (0,0.9) {$W^c(x)$};
			
			% --- Seta f ---
			\draw[thick, ->,>=stealth] (1.8,0) -- (3.2,0)
			node[midway, above] {$f$};

			% --- Folheações em f(x) ---
			\draw[stable, thick]
			plot [smooth] coordinates {(4,-0.4) (4.6,-0.1) (5,0) (5.4,0.1) (6,0.4)};
			\draw[unstable, thick]
			plot [smooth] coordinates {(4,0.4) (4.6,0.1) (5,0) (5.4,-0.1) (6,-0.4)};
			\draw[center, thick]
			plot [smooth] coordinates {(5,-1) (5,-0.3) (5,0) (5,0.3) (5,1)};
			
			% --- Rótulos ---
			\node[unstable, below right] at (5.6,-0.3) {$W^u(f(x))$};
			\node[stable, above right] at (5.6,0.25) {$W^s(f(x))$};
			\node[center, above] at (5,0.9) {$W^c(f(x))$};
			
			\fill (0,0) circle (2pt);
			\node[below left] at (0,0) {$x$};
			% --- Ponto f(x) ---
			\fill (5,0) circle (2pt);
			\node[below right] at (4.9,-0.1) {$f(x)$};
			
		\end{tikzpicture}
		\caption{\label{fig.variedadesinv} Local geometrical picture of a partially hyperbolic diffeomorphism. In some cases, like those treated in Theorems~\ref{teo.main} and \ref{main.teob} the bundle $E^c$ also uniquely integrates.}
	\end{figure}
	
	The topological condition appearing in the statements of Theorems~\ref{main.technical} and~\ref{main.teob} is the following. 
	
	\begin{definition}
		Let $f:M\to M$ be a partially hyperbolic diffeomorphism. We say that $f$ has \emph{minimal stable (resp. unstable) foliation} if every leaf $\cW^s(x)$ (resp. $\cW^u(x)$) is dense in $M$. 
	\end{definition}
	
	\subsection{Subordinated partitions}
	
	Let us discuss now some ergodic theory of a map $f$ as above. We denote by $\cp$ the space of ergodic $f$-invariant measures endowed with the weak$^{\star}$-topology. Fix, for the discussion in the sequel, an element $\mu\in\cp$. 
	
	Take $\xi$ a measurable partition of $M$. Let $\xi(x)$ denote the atom of the partition containing the point $x$. We denote by $\{\mu^{\xi}_x\}_{x\in M}$ the associated family of conditional measures given by Rokhlin's theorem \cite{EinsWar}. 
	
	An important example of application of this is when we consider a foliated box $B\subset M$ with respect to $\cW^u$, i.e. a neighbourhood in $M$ in restriction to which the patition into plaques - connected components of $\cW^u(x)\cap B$ -  is homemorphic to the decomposition of $\R^d$ into affine subspaces of dimension $\dim E^u$. In this case, the plaques of the foliation form a measurable partition of $B$, with respect to $\mu|_{B}$, and we can therefore define the conditional measures $\{\mu^u_{B,x}\}_{x\in B}$ which are supported on the plaques. 
	
	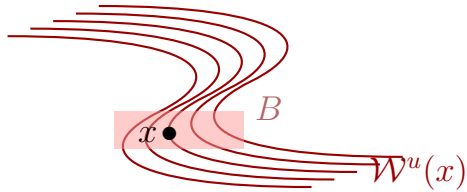
\begin{figure}[h!]
		\centering
		\begin{tikzpicture}
			\draw[red!60!black, thick] (-2,1).. controls (.5,1) and (1,.5) .. (0,0) .. controls (-1,-.5) and (-.5,-1) .. (2,-1) node[below right]{$\cW^u(x)$};
			\begin{scope} [xshift=-.3cm, yshift=-0.1cm]
				\draw[red!60!black, thick] (-2,1).. controls (.5,1) and (1,.5) .. (0,0) .. controls (-1,-.5) and (-.5,-1) .. (2,-1);
			\end{scope}
			\begin{scope} [xshift=-.6cm, yshift=-0.2cm]
				\draw[red!60!black, thick] (-2,1).. controls (.5,1) and (1,.5) .. (0,0) .. controls (-1,-.5) and (-.5,-1) .. (2,-1);
			\end{scope}
			\begin{scope} [xshift=.3cm, yshift=0.1cm]
				\draw[red!60!black, thick] (-2,1).. controls (.5,1) and (1,.5) .. (0,0) .. controls (-1,-.5) and (-.5,-1) .. (2,-1);
			\end{scope}
			\begin{scope} [xshift=.6cm, yshift=0.2cm]
				\draw[red!60!black, thick] (-2,1).. controls (.5,1) and (1,.5) .. (0,0) .. controls (-1,-.5) and (-.5,-1) .. (2,-1);
			\end{scope}
			\fill[red!40!white, opacity=.5] (-1.2,-.7) rectangle (0.5,-0.2);
			\draw (-.47,-.5) node{$\bullet$};
			\draw (-.47,-.5) node[left]{$x$};
			\draw[red!50!black, opacity =.6] (0.5,-0.15) node[right]{$B$};
		\end{tikzpicture}
		\caption{\label{fig.rokhlin}For a Borel set $A\subset B$ we have $\mu(A)=\int_M \mu_{B,x}^u(A)\, d\hat\mu(x)$.}
	\end{figure}
	
	\begin{definition}
		We say that $\mu\in\cp$ is a $u$-Gibbs measure if, for every foliated box $B\subset M$, for $\mu$ almost every $x\in B$, we have $\mu^u_x<<\operatorname{Leb}^u_x$, where $\operatorname{Leb}^u_x$ stands for the inner volume measure on the leaf, defined with the inherited Riemannian structure.   
	\end{definition}
	
	Another approach towards conditional measures on unstable manifolds is to consider the following notion.
	
	\begin{definition}[Subordinated Partition]\label{def:subordinated}
		
		Given $\delta>0$, a measurable partition $\xi$ is said to be $\delta$-subordinated to the unstable foliation $\cW^u$ with respect to the measure $\mu$ if
		\begin{enumerate}
			
			\item For $\mu$-almost every $x\in M$ there is a number $r(x)>0$ such that $\cW_{r(x)}^u(x)\subseteq\xi(x)$.
			
			\item For $\mu$-almost every $x\in M$ we have $\xi(x)\subseteq \cW_{\delta}^u(x)$.
			
			\item $\bigvee_{n=0}^{+\infty}f^{-n}(\xi)$ is the point partition of $M$.
			
			\item $\xi$ is increasing: $\xi\prec f^{-1}(\xi)$.
			
		\end{enumerate}
		
	\end{definition}
	
	The main result of \cite{LedStr} implies the following. See also \cite{Brown_ensaios}, Appendix D.
	
	\begin{theorem}\label{teo.subordinate}
		If $f:M\to M$ is partially hyperbolic and $\mu\in\cp$, then, for every $\delta>0$, there exists a partition $\xi$ which is $\delta$-subordinated to the unstable foliation with respect to $\mu$.
	\end{theorem}
	
	Subordinated partitions will play a very important role in one of the constructions of this paper. 
	
	A consequence of \cite{LedStr} is that in order to speak about $u$-Gibbs measures one may choose to work with subordinated partitions instead of with foliated boxes. 
	
	\begin{proposition}
		If $f:M\to M$ is partially hyperbolic then $\mu\in\cp$ is a $u$-Gibbs measure if, and only if, the conditional measures $\{\mu^{\xi}_x\}_{x\in M}$ with respect to some subordinated partition satisfy $\mu^{\xi}_x<<\operatorname{Leb}^u_x|_{\xi(x)}$. 
	\end{proposition}
	
	%\begin{remark}
	%	It is a consequence of the superposition principle (\textcolor{red}{gotta be careful here}) that if $\mu\in G^u_{erg}(f)$, then the disintegration of $\mu$ in $\xi$ has absolute continuous conditionals with respect to the Lebesgue measure in unstable leaves.
	%\end{remark}
	
	\subsection{The dynamical density function}
	
	Let $f:M\to M$ be a partially hyperbolic diffeomorphism. We shall consider the \emph{dynamical density function} defined for every $y\in\cW^u(x)$ by
	\[
	\rho^u_x(y)\eqdef\lim_{n\to\infty}\frac{\jac^u_x(-n)}{\jac^u_y(-n)}.
	\]
	\begin{remark}
		\label{rem.concicledensity}
		Notice the following cocycle property of the dynamical density function: 
		$\rho^u_{x}(y)\rho^u_{y}(z)=\rho^u_x(z)$, which follows directly from definition.
	\end{remark}
	This function is crutial in the study of $u$-Gibbs measures, as demonstrated by the following fundamental result of Ledrappier \cite{L}:
	\begin{proposition}
		\label{prop.ledra}
	Let $f:M\to M$ be a partially hyperbolic diffeomorphism. Let $\mu\in\cp$ be a $u$-Gibbs measure. Given any subordinated partition $\xi<\cW^u$, for $\mu$-almost every point $x\in M$ the Radon-Nykodim derivative 
	\[
	g^{\xi}_x\eqdef\frac{d\mu^{\xi}_x}{d\Leb^u_x}
	\]
	satisfies, for $\mu^{\xi}_x$-almost every $y\in\xi(x)$, that 
	\[
	\frac{g^{\xi}_x(y)}{g^{\xi}_x(x)}=\rho^u_x(y).
	\]
	\end{proposition}
	This shows that the density of the conditional measures are given, almost surely and up to a measurably varying constant, by the dynamical density function. 
	
	\subsection{Jacobian Periodic Data}

   We denote by $\per$ the set of $p\in M$ such that $f^n(p)=p$, for some integer $n>0$. We denote by $\pi(p)\eqdef\min\{n>0;f^n(p)=p\}$ the period of the periodic orbit $O(p)$.
	
	Given a periodic point $p\in\operatorname{Per}(f)$, we define
	\[
	\jac^u(p)\eqdef\left(\jac^u_p(\pi(p))\right)^{1/\pi(p)}.
	\]
	Notice that $\log\jac^u(p)$ is the sum of unstable Lyapunov exponents of the measure
	\[
	\mu_p\eqdef\frac{1}{\pi(p)}\sum_{\ell=0}^{\pi(p)-1}\delta_{f^\ell(p)}.
	\] 
	\begin{definition}
		We say that a partially hyperbolic diffeomorphism $f:M\to M$ has \emph{constant unstable Jacobian periodic data} if $\jac^u(p)=\jac^u(q)$, for every $p,q\in\per$.
	\end{definition}
	Notice that constant unstable Jacobian periodic data, in our setting, is equivalent to all the measures $\{\mu_p\}_{p\in\per}$ having the same sum of unstable Lyapunov exponents. 
	
	As a consequence of Livschtiz's theorem we have the following:
	
	\begin{proposition}
		\label{prop.cohomologia}
		Let $f:M\to M$ be a $C^2$ partially hyperbolic diffeomorphism. Assume that $f$ has expanding center. Then, $f$ has constant unstable Jacobian periodic data if, and only if, there exists a constant $c\in\R$ and a Hölder continuous function $u:M\to M$ such that 
		\[
		u\circ f(x)-u(x)+c=\log\jac^u_x,
		\]
		for all $x\in M$.
	\end{proposition}
	
	When the conclusion of the above theorem holds we say that \emph{the cocycle $x\mapsto\jac^u_x$ is co-homologous to a constant}. In particular, we could have written the conclusion of Theorem~\ref{teo.main} in this language of cocycles. 
	
	\subsubsection{Joint integrability}
	
	We assume now that $M$ has dimension three, all three bundles $E^s$, $E^c$ and $E^u$ are one-dimensional, and that $f$ has expanding center. In this case $M=\TT$ and $f$ is an Anosov diffeomorphism. Notice that all three bundles in this case have to be one-dimensional. Moreover, in this case the bundle $E^c\oplus E^u$, is uniquely integrable into a two-dimensional expanding $f$-invariant foliation $\cW^{cu}$.  
	
	We define similarly as above the center periodic data: given $p\in\per$ we define 
	\[
	\jac^c(p)\eqdef\left(\jac^c_p(\pi(p))\right)^{1/\pi(p)}.
	\]
	Similarly we say that $f$ has constant center periodic data if $\jac^c(p)=\jac^c(q)$, for every $p,q\in\per$. 
	It follows from \cite{GanShi} that
	
	\begin{theorem}
		\label{teo.ganshi}
		Let $f:M\to M$ be a three dimensional partially hyperbolic Anosov diffeomorphism with expanding center. Then, $f$ has constant center periodic data if, and only if, the two dimensional bundle $E^s\oplus E^u$ is uniquely integrable.
	\end{theorem}
	
	\begin{definition}
		When the two dimensional bundle $E^s\oplus E^u$ is uniquely integrable we say that $f$ is \emph{jointly integrable}.	
	\end{definition}
	
	Combining Gan-Shi's result \cite{GanShi} with the arguments in \cite{GG} we obtain the following result.
	
	\begin{proposition}
		\label{prop.jointlyintegrable}
		Let $f:M\to M$ be a three dimensional partially hyperbolic Anosov diffeomorphism with expanding center. Assume further that $f$ is jointly integrable and has constant unstable periodic data. Then, $f$ is conjugated to a linear Anosov map $L:\TT\to\TT$ by a conjugacy $h$ which is $C^{1+\alpha}$ in restriction to the center unstable manifold $\cW^{cu}$.  
	\end{proposition}
	
	Recall that two diffeomorphisms $f:M\to M$ and $g:N\to N$ are said to be conjugate if there exists a homeomorphism $h:M\to N$ such that $g\circ h=h\circ f$. We stress that the existence of the conjugacy in the above result is a classical fact, following from \cite{Franks} and \cite{casanova}. The non-trivial statement is the smoothness of the conjugacy in restriction to the center-unstable foliation $\cW^{cu}$.
	
	\subsection{Unstable entropy}
	
	We return now to the more general setting where $f:M\to M$ is a partially hyperbolic diffeomorphism, with no assumptions on the dimension of the invariant sub-bundles nor on the behaviour of $f$ along the center. 
	
	Given $y\in\cW^u(x)$ define the \emph{dynamical unstable distance} by
	\[
	d^u_n(x,y)=\max_{0\leq j\leq n-1}d^u(f^jx,f^jy),
	\] 
	where $d^u(.,.)$ is the induced inner Riemannian metric on the leaf $\cW^u(x)$. We say that a set $E\subset\cW^u_{\delta}(x)$ is $(u,n,\eps)$-separated if for every $y,\hat y\in E$ one has $d^u_n(y,\hat y)>\eps$. Denote by $N^u(\eps,x,\delta,n)$ the maximal possible cardinality of a $(u,n,\eps)$-separated set within $\cW^u_{\delta}(x)$. 
	
	\begin{definition}
		The \emph{unstable topological entropy} is defined by
		\[
		h^u(f)\eqdef\lim_{\delta\to 0}\sup_{x\in\TT}\lim_{\eps\to 0}\limsup_{n\to\infty}\frac{1}{n}\log N^u(\eps,x,\delta,n).
		\]
	\end{definition}  
	
	Although we do not use in our work, it is interesting to remark that in \cite{HHW} the authors prove that 
	\[
	h^u(f)=\sup_{x\in M}\limsup_{n\to\infty}\frac{1}{n}\log\Leb^u_x\left(f^n(\cW^u_{\delta}(x))\right),
	\]
	for every $\delta>0$. See Theorem C in \cite{HHW}.
	
	We now turn to the measure-theoretical version of the above concept. First, recall that given $\xi,\eta$ measurable partitions, the conditional entropy with respect to a measure $\mu$ is the quantity
	\[
	H_{\mu}(\xi|\eta)\eqdef-\int_M\log(\mu^{\eta}_x(\xi(x)))d\mu(x)
	\] 
	Given $\xi<\cW^u$ and $\eta<\cW^u$ two subordinate partitions one can show from the work of Ledrappier-Young \cite{LYI} that  $H_{\mu}(\xi|f\xi)=H_{\mu}(\eta|f\eta)$. A sketch of proof can be found in \cite{Ali}, page 84 (with precise references to \cite{LYI}). This allows us to define to following key notion.
	
	\begin{definition}
		Given $\xi<\cW^u$ the \emph{unstable metric entropy} of the system $(f,\mu)$ is the number $h_{\mu}^u(f)\eqdef H_{\mu}(\xi|f\xi)$	
	\end{definition}	
	
	\begin{remark}
		The proofs of \cite{LedStr,L} show that an element	$\mu\in\cp$ is an $u$-Gibbs measure if, and only if $h^u_{\mu}(f)=\int\log\|Df(x)|_{E^u}\|d\mu(x)$.	
	\end{remark}

	In \cite{HHW}, Theorem D, the authors establish the following unstable variational principle.
	
	\begin{theorem}[Hu-Hua-Wu]
		$h^u(f)=\sup\{h^u_{\mu}(f);\mu\in\cp\}$
	\end{theorem}	
	
	This result leads us to consider measures which attains the maximum value of unstable metric entropy.
	
	\begin{definition}
		We say that $\mu\in\cp$ is a \emph{measure of maximal unstable entropy} if $h^u_{\mu}(f)=h^u(f)$. 
	\end{definition}
	
	\subsection{Systems of unstable measures}
	
	In this work we can prove a general rigidity theorem under a very flexible assumption which we now describe. So, let $f:M\to M$ be a partially hyperbolic diffeomorphism. 
	
	\begin{definition}
		We say that a measure $\mu\in\cp$ admits a \emph{Margulis system of unstable measures} if there exists a measurably varying family $\{m^u_x\}_{x\in M}$ of locally finite Borel measures on the leaf $\cW^u(x)$ such that
		\begin{enumerate}
			\item If $y\in\cW^u(x)$ then $m^u_x=m^u_y$.
			\item For every foliated box $B\subset M$ there exists a constant $C>0$ such that the disintegration $\{\mu^B_x\}_{x\in B}$ of $\mu$ with respect to the plaques of $\cW^u(x)\cap B$ satisfies
			\[
			m^u_x|_{B\cap\cW^u(x)}=C\mu^B_x,
			\]
			for $\mu$-almost every $x\in B$. 
			\item There exists $\lambda>0$ such that $m^u_x\circ f=e^{\lambda}m^u_{f^{-1}(x)}$.
		\end{enumerate}  
	\end{definition}
	
	With this definition at hand we can re-state our main technical result. 
	
	\begin{theorem}
		\label{teo.maintechnical}
		Let $f:M\to M$ be a partially hyperbolic diffeomorphism. Assume that there exists some $u$-Gibbs measure $\mu\in\cp$ which admits a Margulis system of unstable measures. Then, $f$ has constant unstable Jacobian periodic data.  
	\end{theorem}
	
	\subsection{Proofs of the main results}
	
	In this section we show how to deduce the main theorems presented in the introduction from Theorem~\ref{teo.maintechnical} and the preliminary material we gave.
	
	\subsubsection{Proof of Theorem~\ref{teo.main}}
	We shall now show how to deduce Theorem~\ref{teo.main} from Theorem~\ref{teo.maintechnical}. So we assume that $f:\TT\to\TT$ is a $C^{\infty}$ Anosov partially hyperbolic diffeomorphism with expanding center. Let $\mu\in\cp$ be a $u$-Gibbs measure. Assume that the unstable entropy of $\mu$ equals the unstable topological entropy of $f$, i.e. $h^u_{\mu}(f)=h^u(f)$. Therefore, by Corollary 1.2 of \cite{humbert} we have that $\mu$ admits a Margulis system of unstable measures. This puts us in position to apply directly Theorem~\ref{teo.maintechnical}, deducing that $f$ has constant unstable periodic data. \qed
	
	\subsubsection{Proof of Corollary~\ref{cor.casoji}}
	Let $f:\TT\to\TT$ be a $C^{\infty}$ partially hyperbolic Anosov diffeomorphism with expanding center and jointly integrable. If $f$ admits a $u$-Gibbs measure which is also a measure of maximal unstable entropy, by Theorem~\ref{teo.main} we conclude that $f$ has constant unstable periodic data. Applying Proposition~\ref{prop.jointlyintegrable} we conclude that $f$ is conjugated to a linear Anosov map $L:\TT\to\TT$ by a conjugacy $h$ which is $C^{1+\alpha}$ in restriction to the center unstable manifold $\cW^{cu}$. The uniqueness of $u$-Gibbs measure in this case follows since a conjugacy which is $C^{1+\alpha}$ along the foliation $\cW^{cu}$ sends $u$-Gibbs measures of $f$ into $u$-Gibbs measures of the linear Anosov map $L:\TT\to\TT$ and $L$ has a unique $u$-Gibbs measure, which is the Haar measure of $\TT$. \qed

	\subsubsection{Proof of Theorem~\ref{main.teob}}
	
	The proof of Theorem~\ref{main.teob} is identical to that of Theorem~\ref{teo.main}: one only has to notice that, for $C^2$ center isometries the existence of a Margulis system of unstable measures is given in Theorem A of \cite{CRH}. \qed 
	
	The rest of this paper is devoted to the proof of Theorem~\ref{teo.maintechnical}.

	\section{Conditional measures on entire unstable leaves}\label{sec.leafwise}

	In this section we develop in detail a key technical ingredient of our proof, which is a canonical realization of Rokhlin conditional measures as locally finite measures defined on entire unstable leaves. Although Rokhlin conditional measures are the standard object in smooth ergodic theory, they are only defined on atoms of measurable subordinate partitions. For the arguments developed in this paper one needs a canonical family of measures defined on the entire unstable leaf, varying measurably with the base point and compatible with the Rokhlin disintegration. To the best of our knowledge, this construction has not appeared in this level of generality. Although inspired by the stabilization procedure introduced in \cite{ALOS} (which itself is inspired in earlies arguments from \cite{KalininKatok} and \cite{KatokSpa}) the object constructed here is fundamentally different: instead of quotient measures arising from center normal forms, we construct canonical leafwise realizations of unstable conditional measures for arbitrary partially hyperbolic diffeomorphisms. The advantage of working with these leafwise measures is that one can express the unstable Jacobian cocycle directly through global Radon measures on unstable leaves, avoiding repeated changes of subordinate partitions throughout the proof. 
	
	We also stress that, differently from \cite{ALOS}, the argument we present here makes no use of normal forms or any kind of non-stationary linearization of the dynamics. 
	
%	In this section  we will describe a general procedure to construct conditional measures supported on entire leafs, even when the partition into unstable leaves is not measurable. The inevitable drawback is that the conditional measures are no longer probabilities, but they coincide with the usual disintegration up to normalization. This general procedure is sketched in the first chapter of \cite{Brown_ensaios}. The same idea was used in \cite{ALOS} for the construction of the so-called leafwise quotient measures. Thus, our arguments follows very closely those of \cite{ALOS}.  Nevertheless we choose to give a detailed treatment, since we believe that the general construction of leafwise measures along unstable manifolds has an independent interest.
	
	In the end of the section we specialise to the case of a $u$-Gibbs measure and we describe how the leafwise measure change when we push the base point under the dynamics.

	\subsection{Construction of the Leafwise Measures}
	
	Let $f:M\to M$ be a partially hyperbolic diffeomorphism. The goal of this section is to give a detailed proof of the result below.
	
	\begin{theorem}[Leafwise Measures]\label{teo:leafwise}
		For $\mu\in \cP^{\operatorname{erg}}_f(M)$, there exists a family of measures $\{\mu_x^u\}_{x\in M}$ in unstable leafs (defined for $\mu$ almost every point and depending measurably on the base point) such that, if $\xi<\cW^u$ is a subordinated partition, then for $\mu$-a.e.$x\in M$
		\[
		\mu_x^\xi=\frac{\mu_x^u|_{\xi(x)}}{\mu_x^u(\xi(x))}
		\]
	\end{theorem}
	
	We call the measure $\mu^u_x$ the \emph{leafwise measure along the unstable manifold} $\cW^u(x)$. 
	
	The outline of the proof, which, as said earlier, follows very closely the arguments in \cite{ALOS} is that we start with a subordinated partition (definition \ref{def:subordinated}) and we iterate it to obtain partitions with bigger and bigger atoms. By disintegrating the measure $\mu$ in each of these partitions, we obtain families of measures defined in unstable leafs whose support grows arbitrarily big (Lemma \ref{lem.growth}). We can show that up to a non-stationary normalization in space, they all agree after some time. In the limit we obtain our desired family.	
	
	\subsubsection{Constructing the family $\{\mu_x^u\}_{x\in M}$}
	
	Let $\xi_0<\cW^u$ be a subordinated partition and for each $n\in\Z$ define
	\[
	\xi_n\eqdef f^n{\xi_0}
	\]
	By item (iv) of definition \ref{def:subordinated} the sequence
	\[
	\cdots\prec\xi_n\prec\cdots\prec\xi_0\prec\cdots\prec\xi_{-n}\prec\cdots
	\]
	is increasing. By Rokhlin's Theorem, we can consider the disintegration $\{\mu_x^{\xi_n}\}_{x\in M}$ of $\mu$ with respect to each of these partitions.
	
	%The following is a foreshadowing of the superposition principle that is to come later
	%
	%\begin{theorem}
	%	If $m>n$ then, for $\mu$-a.e.$x\in\TT$, for all $I\in\xi(x)$ the following formula holds
	%		\[
	%		\mu_x^{\xi_m}(A)=\sum_{\xi_n(y)\subseteq\xi_m(x)}\mu_y^{\xi_n}(A\cap\xi(y))
	%		\]
	%\end{theorem}
	%\begin{proof}
	%	Fist, I must highlight that the sum on the right is taken over every atom of $\xi_n$ that is contained in $\xi_m(x)$; and, the term $\mu_y^{\xi_n}$, is well defined independently of the representative for $\xi(y)$; because if $z\in\xi(y)$ we have $\mu_y^{\xi_n}=\mu_z^{\xi_n}$.
	%
	%	To prove it, just notice that 
	%\end{proof}
	
	The next lemma shows that the domain of almost every $\mu_x^{\xi_n}$ grows arbitrarily in $\cW^u(x)$:
	
	\begin{lemma}\label{lem.growth}
		For $\mu$-almost every point $x\in M$ and every $R>0$, $\exists n_0\in\N$ such that
		\[
		\cW^u_R(x)\subseteq\xi_n(x)
		\]
		for all $n>n_0$.
	\end{lemma}
	\begin{proof}
		For $\varepsilon>0$ define the set
		\[
		A(\varepsilon)\eqdef\{x\in M|\cW^u_\varepsilon(x)\subseteq\xi(x)\}
		\]
		If $x\in\TT$ is such that $f^{-n}(x)\in A\left(\frac{1}{k}\right)$, then, 
		\[
		f^n\cW^u_\frac{1}{k}(f^{-n}(x))\subseteq f^n(\xi(f^{-n}(x))).
		\]
		Consider $\beta\eqdef\inf_{x\in M}\|Df(x)|_{E^u}\|$. Partial hyperbolicity of $f$ gives that $\beta>1$. Thus $\cW^u_{\frac{\beta^n}{k}}(x)\subseteq f^n\cW^u_\frac{1}{k}(f^{-n}(x))$ and the equation above gives that
		\[
		\cW^u_{\frac{\beta^n}{k}}(x)\subseteq\xi_n(x)
		\]
		Hence, if $n>\log_{\beta}(kR)$ we have that $\cW^u_R(x)\subseteq\xi_n(x)$ and the lemma is true for this $x$. Therefore, it suffices to prove that the set
		\begin{equation}
			\tag{\textasteriskcentered}
			\bigcup_{\overset{k\in\N}{n>\log_\lambda(kR)}}f^nA\left(\tfrac{1}{k}\right)
		\end{equation}
		has full measure. Indeed, notice that, by item (i) of \ref{def:subordinated}, the set $\bigcup_{k\in\N}A\left(\tfrac{1}{k}\right)$ has full measure. So, by the $f$-invariance of $\mu$, the same is true for (\textasteriskcentered), concluding.
	\end{proof}
	
	\begin{lemma}\label{lem.superposicaolindenstrauss}
		For every $m>n\geq 0$ and for $\mu$-almost every $x\in M$,  it holds that $\mu^{\xi_m}(\xi_n(x))>0$.
	\end{lemma}	
	\begin{proof}
		Fix $m>n$. Consider the set
		\[
		Y\eqdef\{x\in M|\mu_x^{\xi_m}(\xi_n(x))=0\}.
		\]
		We claim that $\mu_x^{\xi_n}(Y)=0$ for $\mu$-a.e.$x\in M$. Using that $\{\mu_x^{\xi_n}\}_{x\in M}$ is a disintegration of $\mu$ one deduces that $Y$ has 0 measure, which is the desired conclusion. Therefore, it only remains to prove our claim.
		
		To establish the claim, use that by item (1) of Definition \ref{def:subordinated}, for almost every $x\in M$, we can cover $\mu_x^{\xi_m}$-almost every point of $\xi_m$ with at most countably many atoms of $\xi_n$. Hence, we can write
		\[
		\mu_x^{\xi_m}(Y)=\sum_{\xi_n(y)\subseteq\xi_m(x)}\mu_x^{\xi_m}(\xi_n(y)\cap Y)
		\]
		where the sum on the right is taken over the countably many atoms of $\xi_n$ that are contained in $\xi_m(x)$. We can show that every term on the right side of this formula is 0: if $\xi_n(y)\cap Y=\emptyset$, then there is nothing to verify; if that intersection in non empty, let $z\in\xi_n(y)\cap Y$. By definition of $Y$, we have $\mu_z^{\xi_m}(\xi_n(y)\cap Y)=0$; however, since $z\in\xi_n(y)$ and $\xi_n(y)\subseteq\xi_m(x)$, we know that $z$ and $x$ have the same atom in $\xi_m$, thus $\mu_x^{\xi_m}=\mu_z^{\xi_m}$ and we obtain that $\mu_x^{\xi_m}(\xi_n(y)\cap Y)=0$ which concludes the proof.
	\end{proof}
	
	Now, for each point $x\in M$, we we'll fix a reference interval $J_x\subseteq \cW^u(x)$ by defining $J_x\eqdef\cW^u_1(x)$. Using Lemma~\ref{lem.growth} we know that, for $\mu$-a.e.$x\in M$, there is a positive integer $n_x\in\N$ such that for all $n>n_x$ the domain of $\mu_x^{\xi_n}$ contains $J_x$.
	
	It is with this interval that we will normalize the sequence $\{\mu_x^{\xi_n}\}_{n\in\N}$. The next lemma clarifies why we can do that.
	
	\begin{lemma}\label{lem.positive}
		There exists  subordinated partition $\xi_0$ such that for $\mu$-a.e.$x\in M$ and for all $n>0$ it holds that $\mu_x^{\xi_n}(J_x)>0$.
	\end{lemma}
	\begin{proof}
	 Let $\xi_0$ be a $1$-subordinated partition, as given by Theorem~\ref{teo.subordinate}. This implies that $\xi_0(x)\subseteq J_x$ for all $x\in M$. For each $n$, Lemma~\ref{lem.superposicaolindenstrauss} gives a full measure subset of $x\in M$ satisfying $\mu^{\xi_n}_x(\xi_0(x))>0$. Taking the intersection of these countably many sets of full measure and using $\xi_0(x)\subseteq J_x$ for all $x\in M$ we conclude.
	\end{proof}

	With this lemma, we can define for almost every $x\in M$ the measures
	\[
	\frac{\mu_x^{\xi_n}}{\mu_x^{\xi_n}(J_x)}
	\]
	who all give the same value to $J_x$. In fact, we will show that they in fact agree in every set as $n\to+\infty$ (Lemma \ref{lem.sebaorigin}). The fundamental step in this direction is the following superposition property.
	
	\begin{lemma}\label{lem.sebafundamental}
		For $\mu$-a.e.$x\in M$, if $m>n$ then for all $A\subseteq\xi_n(x)$ measurable, we have
		\[
		\mu_x^{\xi_m}(A)=\mu_x^{\xi_n}(A)\mu_x^{\xi_m}(\xi_n(x))
		\]
	\end{lemma}
	\begin{proof}
		Define an auxiliary measure $\eta_{x,n}^{\xi_m}$ in $\xi_m(x)$ by
		\[
		\eta_{x,n}^{\xi_m}(A)\eqdef\int_{\xi_m(x)}\left[\int_{\xi_n(y)}\mathds{1}_A(z)d\mu_y^{\xi_n}(z)\right]d\mu_x^{\xi_m}(y)
		\]
		Let $\varphi:M\to\R$ be an arbitrary integrable function and define
		\[
		\psi(q)\eqdef\int_{\xi_n(q)}\varphi(z) d\mu_q^{\xi_n}(z)
		\]
		Using that both $\{\mu_x^{\xi_m}\}_{x\in M}$ and $\{\mu_x^{\xi_n}\}_{x\in M}$ are disintegrations of $\mu$ we can use Rokhlin's Theorem twice to see that
		\begingroup
		\renewcommand{\arraystretch}{2.5}
		\[\begin{array}{rl}
			\displaystyle\int\left[\int_{\xi_m(p)}\varphi d\eta_{p,n}^{\xi_m}\right]d\mu(p)&\displaystyle=\int\left[\int_{\xi_m(p)}\left[\int_{\xi_n(q)}\varphi(z) d\mu_q^{\xi_n}(z)\right]d\mu_p^{\xi_m}(q)\right]d\mu(p)\\&\displaystyle=\int\left[\int_{\xi_m(p)}\psi(q)d\mu_p^{\xi_m}(q)\right]d\mu(p)\\&\displaystyle=\int\psi(p)d\mu(p)\\
			&\displaystyle=\int\left[\int_{\xi_n(p)}\varphi d\mu_p^{\xi_n}\right]d\mu(p)\\
			&\displaystyle=\int\varphi d\mu
		\end{array}\]
		\endgroup
		Hence, $\{\eta_{x,n}^{\xi_m}\}_{x\in M}$ is a disintegration of $\mu$ with respect to $\xi_m$. By uniqueness of the disintegration, we have $\eta_{x,n}^{\xi_m}=\mu_x^{\xi_m}$ for $\mu$-a.e.$x\in M$. Now, for $m>n$ and $A\subseteq\xi_n(x)\subseteq\xi_m$ we have
		\begin{align*}
			\mu_x^{\xi_m}(A)&=\eta_{x,n}^{\xi_m}(A)=\int_{\xi_m(x)}\left[\mathds{1}_{\xi_n(x)}(y)\int_{\xi_n(y)}\mathds{1}_A(z)d\mu_y^{\xi_n}(z)\right]d\mu_x^{\xi_m}(y)\\
			&=\mu_x^{\xi_n}(A)\mu_x^{\xi_m}(\xi_n(x)).\qedhere
		\end{align*}
	\end{proof}
	
	This lemma, which is  a foreshadowing of the superposition principle that is to come later, can be re-stated in a much more convenient way:
	
	\begin{lemma}\label{lem.sebaorigin}
		For $\mu$-a.e.$x\in M$, it holds that for every compact $K\subseteq W^u(x)$ there exists a $n_0\in\N$ such that for all $m,n>n_0$
		\[
		\frac{\mu_x^{\xi_m}(K)}{\mu_x^{\xi_m}(J_x)}=\frac{\mu_x^{\xi_n}(K)}{\mu_x^{\xi_n}(J_x)}
		\]
	\end{lemma}
	\begin{proof}
		Since $K$ is compact -hence bounded- Lemma \ref{lem.growth} says that we can find a $n_0\in\N$ such that $K\subseteq\xi_k(x)$ for all $k>n_0$, thus the numerators above are well defined. Also, by Lemma~\ref{lem.superposicaolindenstrauss}, for $\mu$-a.e.$x\in M$ the denominators above are well defined for $m,n>0$. Hence, aplying lemma \ref{lem.sebafundamental} for $K$ and $J_x$ in place of $A$, we get
		\[
		\frac{\mu_x^{\xi_m}(K)}{\mu_x^{\xi_m}(J_x)}=\frac{\mu_x^{\xi_n}(K)\mu_x^{\xi_m}(\xi_n(x))}{\mu_x^{\xi_n}(J_x)\mu_x^{\xi_m}(\xi_n(x))}=\frac{\mu_x^{\xi_n}(K)}{\mu_x^{\xi_n}(J_x)}
		\]
		as desired.
	\end{proof}
	
	Since all these measures agree, we can define an additive function
	\[
	\mu_x^u(K)\eqdef\lim_{n\to+\infty}\frac{\mu_x^{\xi_n}(K)}{\mu_x^{\xi_n}(J_x)}
	\]
	which, by Lemma \ref{lem.growth}, is defined for every compact set $K\subseteq W^u(x)$. This function induces a borelian measure -also denoted by- $\mu_x^u$ on the unstable leaf of $x$. We claim that this family $\{\mu_x^u\}_{x\in M}$ satisfy Theorem \ref{teo:leafwise}.
	
	\subsubsection{Proof of Theorem \ref{teo:leafwise}}
	Let $\{\mu_x^u\}_{x\in M}$ be the family constructed above, fix a subordinated partition $\xi< \cW^u$.
	
	To prove Theorem \ref{teo:leafwise}, we need the following general lemma, which is a direct computation, see \cite{ALOS} lemma 2.12.
	
	\begin{lemma}\label{lem.partiiterado}
		For $\mu$-a.e.$x\in M$ and $n\in\N$ we have
		\[
		\mu_x^{f^n\xi}=f^n_*\mu_{f^{-n}(x)}^{\xi}
		\]
	\end{lemma}
	
	With that and an approximation argument we can prove the superposition principle:
	
	\begin{lemma}
		If $\xi<\cW^u$ is any partition subordinated to $\cW^u$, then for $\mu$-a.e.$x\in M$
		\[
		\mu_x^\xi=\frac{\mu_x^u|_{\xi(x)}}{\mu_x^u(\xi(x))}
		\]
	\end{lemma}
	\begin{proof}
		Let $\xi_0$ be the subordinated partition used in the construction of the measures $\mu_x^u$. By item $(1)$ of definition \ref{def:subordinated}, we can find an increasing family of subsets $K_1\subseteq K_2\subseteq\cdots\subseteq K_n\subseteq\cdots\subseteq\TT$ such that $\mu(K_i)\to1$ and, for $i\in\N$, there is a $\varepsilon_i>0$ such that $W_{\varepsilon_i}^u(x)\subseteq\xi_0(x)$ for all $x\in K_i$.
		
		Since atoms of $\xi$ are uniformly bounded by a constant $\delta>0$, for each $i$ let $n_i\in\N$ be big enough so that $\beta^{-n_i}\delta<\varepsilon_i$. With that choice, for every $p\in K_i$, the atom of $f^{-n_i}\xi$ at $x$ is contained in the atom $\xi_0(x)$ of $\xi_0$. Lets restrict $f^{-n_i}\xi$ and $\xi_0$ to partitions $f^{-n_i}\xi|_{K_i}$ and $\xi_0|_{K_i}$ of $K_i$.
		
		By Lemmas \ref{lem.sebafundamental} and \ref{lem.positive}, for $\mu$-a.e.$x\in K_i$ we have
		\[
		\mu_x^{f^{-n_i}\xi|_{K_i}}=\frac{\mu_x^{\xi_0|_{K_i}}|_{f^{-n_i}\xi|_{K_i}}}{\mu_x^{\xi_0|_{K_i}}(f^{-n_i}\xi|_{K_i}(x))}
		\]
		Using Lemma \ref{lem.partiiterado}, we have
		\[
		\mu^{f^{-n_i}\xi|_{K_i}}_x=f^{-n_i}_*\mu^{\xi|_{K_i}}_{f^{n_i}x}
		\]
		and
		\[
		\mu^{\xi_0|_{K_i}}_x=f^{-n_i}_*\mu^{\xi_{n_i}|_{K_i}}_{f^{n_i}x}
		\]
		Thus
		\[
		f^{-n_i}_*\mu^{\xi|_{K_i}}_{f^{n_i}x}=\frac{f^{-n_i}_*\mu^{\xi_{n_i}|_{K_i}}_{f^{n_i}x}|_{f^{-n_i}\xi|_{K_i}}}{f^{-n_i}_*\mu^{\xi_{n_i}|_{K_i}}_{f^{n_i}x}(f^{-n_i}\xi|_{K_i}(x))}
		\]
		Since $f^{-n_i}\xi|_{K_i}(x)=K_i\cap f^{-n_i}(\xi(f^{n_i}x)$ and $f^{-n_i}_*\mu(A)=\mu(f^{n_i}A)$, this denominator on the right is just $\mu^{\xi_{n_i}|_{K_i}}_{f^{n_i}x}(f^{n_i}(K_i)\cap\xi(f^{n_i}x))$. Hence, applying $f^{n_i}_*$ and writing $y=f^{n_i}(x)$, we obtain that
		\[
		\mu^{\xi|_{K_i}}_y=\frac{\mu^{\xi_{n_i}|_{K_i}}_y|_{\xi|_{K_i}}}{\mu^{\xi_{n_i}|_{K_i}}_y(f^{n_i}(K_i)\cap\xi(y))}=\frac{\mu^{\xi_{n_i}}_y|_{K_i\cap\xi(y)}}{\mu^{\xi_{n_i}}_y(f^{n_i}(K_i)\cap\xi(y))}
		\]
		for every $y\in f^{n_i}(K_i)$. Since $\bigcup_i K_i$ covers almost every point, the above equality actually holds for the entire atom of $\xi$ almost everywhere and the lemma is proven.
	\end{proof}
	
	\subsection{Application to $u$-Gibbs measures}
	
	The construction above is very general, as it holds for any invariant measure; now we shall prove the following property, which holds specifically for u-Gibbs measures:
	
	\begin{proposition}\label{prop.leafwisegibbs}
		Let $f:M\to M$ be a partially hyperbolic diffeomorphism. Given a $u$-Gibbs measure $\mu\in\cp$ there exists a measurably varying family $\{g_x\}_{x\in M}$ of positive continuous functions $g_x:\cW^u(x)\to\R$ such that for $\mu$-almost every $x\in M$ the following holds:
		\begin{enumerate}
			\item $\mu^u_x(A)=\int_Ag_x(y)d\Leb^u_x(y)$, for every bounded measurable set $A\subset\cW^u(x)$;
			\item For $\mu^u_x$-almost every $y\in\cW^u(x)$ one has
			\[
			\frac{g_x(y)}{g_x(x)}=\rho^u_x(y).
			\]
		\end{enumerate}
\end{proposition}	
\begin{proof}
Let us begin proving the existence of a family of $\Leb^u_x$ integrable functions $g_x:\cW^u(x)\to\R$ satisfying the first item. For that, we let $\xi_0<\cW^u$ be the subordianted partition used in the construction of the leafwise family $\{\mu^u_x\}_{x\in M}$. Consider, as before, $\xi_n=f^n(\xi_0)$, for $n>0$. Observe that we may assume, without lost of generality, that for $\mu$-almost every $x\in M$ and for every $n\in\N$ the conditional measure $\mu^{\xi_n}_x$ is absolutely continuous with respect to $\Leb^u_x$. 

We also know, from the construction of the leafwise family, that for $\mu$-almost every $x\in M$ there exists $n_0(x)\in\N$ such that whenever $n\geq n_0(x)$ one has 
\[
\mu^u_x(A)=\frac{\mu^{\xi_n}_x(A)}{\mu^{\xi_n}(J_x)},
\] 
for every measurable set $A\subset\xi_n(x)$, where $J_x=\cW^u_1(x)$.

Now, fix arbitraly a positive integer $m\in\N$ and a $\mu$-typical point $x\in M$ (satisfying the above properties). Take $n>n_0(x)$ large enough so that 
\[
\cW^u_m(x)\subset\xi_n(x).
\]  
Let $B\subset\cW^u_m(x)$ be a measurable set with $\Leb^u_x(B)=0$. Then, 
\[
\mu^u_x(B)=\frac{\mu^{\xi_n}_x(B)}{\mu^{\xi_n}(J_x)}=0.
\]
Since $m$ is arbitrary, this proves that if $\hat{B}\subset\cW^u(x)$ has $\Leb^u_x(\hat{B})=0$ then $\mu^u_x(\hat{B})=0$. By Radon-Nikodym's theorem, there exists an integrable function $g_x:\cW^u(x)\to\R$ such that 
\[
\mu^u_x(A)=\int_Ag_x(y)d\Leb^u_x(y),
\]
for every measurable set $A\subset\cW^u(x)$. 

We now claim that for $\mu$-almost every $x\in M$ we have
\begin{equation}
\label{eq.ledra}
\frac{g_x(y)}{g_x(x)}=\rho^u_x(y),
\end{equation}
for $\mu^u_x$-almost every $y\in\cW^u(x)$. Notice that, if this claim is true, then the integrable function $g_x$ is actually continuous on the leaf $\cW^u(x)$. This shows that we only have to prove our claim in order to establish completely the proposition. 

To prove the claim, fix $m\in\N$ and $\eps>0$. Given a $\mu$-typical point $x\in M$ and $n>n_0(x)$ sufficiently large so that $\cW^u_m(x)\subset\xi_n(x)$, let $g_x^n\eqdef\frac{d\mu^{\xi_n}}{d\Leb^u_x}$. Given a Borel set $A\subset\cW^u_m(x)$, we know that
\[
\mu^u_x(A)=\frac{\mu^{\xi_n}_x(A)}{\mu^{\xi_n}_x(J_x)}=\frac{1}{\mu^{\xi_n}_x(J_x)}\int_Ag^n_x(y)d\Leb^u_x(y).
\] 
This shows that 
\[
g_x(y)=\frac{g_x^n(y)}{\mu^{\xi_n}_x(J_x)},
\]
for $\Leb^u_x$ almost every $y\in\cW^u_m(x)$. As a consequence, we get a full measure subset $K_m\subset M$ so that whenever $x\in K_m$ there exists a positive integer $n\in\N$ satisfying
\begin{equation}
		\label{eq.gost}
\frac{g_x(y)}{g_x(x)}=\frac{g^k_x(y)}{g^k_x(x)},
\end{equation} 
for $\Leb^u_x$ almost every $y\in\cW_m^u(x)$ and every $k\geq n$. 

Let $x\mapsto \gamma(x)$ be the function which to $\mu$-almost every $x\in M$ associates an integer $n=\gamma(x)$ for which \eqref{eq.gost} holds. Take $N\in\N$ sufficiently large so that $\mu\left(\bigcup_{n=1}^N\gamma^{-1}(n)\right)>1-\eps$.  By regularity of $\mu$ one obtains a compact set $K^{\eps}_m$ with $\mu(K^{\eps}_m)>1-\eps$ such that for every $x\in K^{\eps}_m$, \eqref{eq.gost} holds with $k=N$ for every $y\in\cW^u_m(x)$.  

Now, take $\eps_m=2^{-m}$ and apply Borell-Canteli's lemma. Then, for $\mu$-almost every point $x\in M$ there exists $m(x)$ so that if $m>m(x)$ then $x\in K^{\eps_m}_m$. In particular, \eqref{eq.gost} holds for sufficiently large $k$. Applying Proposition~\ref{prop.ledra} this ensures that   
\[
\frac{g_x(y)}{g_x(x)}=\rho^u_x(y),
\]
for $\Leb^u_x$ almost every $y\in\cW_m^u(x)$ and every $m>m(x)$. This establishes our claim and completes the proof.
\end{proof}

Proposition~\ref{prop.leafwisegibbs} has the following crucial implication for our arguments. The result is largely inspired by one of the \emph{basic moves} for the so-called leafwise quotient measures of \cite{ALOS}. 

\begin{corollary}
	\label{cor.leafwisegibbs}
	Let $f:M\to M$ be a partially hyperbolic diffeomorphism. Given a $u$-Gibbs measure $\mu\in\cp$, there exists a measurably varying family $\{g_x\}_{x\in M}$ of positive continuous functions $g_x:\cW^u(x)\to\R$ such that
	For $\mu$-almost every $x\in M$ the leafwise measures $\{\mu^u_x\}_{x\in M}$ satisfy
	\[
	\mu^u_{f^k(x)}(f^k(A))=\frac{g_{f^k(x)}(f^k(x))}{g_x(x)}\times\jac^u_x(k)\mu^u_x(A),
	\]
	for every bounded measurable set $A\subset\cW^u(x)$ and every $k\in\Z$.
\end{corollary}
\begin{proof}
From Proposition~\ref{prop.leafwisegibbs} and applying the change of variables formula	we get
\[
\mu^u_{f(x)}(f(A))=\int_{f(A)}g_{f(x)}(z)d\Leb^u_{f(x)}(z)=\int_{A}\jac^u_yg_{f(x)}(f(y))d\Leb^u_x(y).
\]
Since 
\[
\rho^u_{f(x)}(f(y))=\frac{\jac^u_x}{\jac^u_y}\rho^u_x(y),
\]
we deduce from the second item in Proposition~\ref{prop.leafwisegibbs} that
\[
g_{f(x)}(f(y))=\frac{g_{f(x)}(f(x))}{g_x(x)}\frac{\jac^u_x}{\jac^u_y}g_x(y).
\]
Therefore,
\begin{align*}
	\mu^u_{f(x)}\left(f(A)\right)&=\int_A\jac^u_x\frac{g_{f(x)}(f(x))}{g_x(x)}g_x(y)d\Leb^u_x(y)\\
	&=\jac^u_x\frac{g_{f(x)}(f(x))}{g_x(x)}\mu^u_x(A),
\end{align*}
as claimed when $k=1$. The result now follows from induction. 
\end{proof}
	\section{A leafwise Jacobian formula}\label{sec.formula}
	
	In this section we develop the core argument of our proof. The goal is to develop a measurable comparison between the Leafwise Measures and the Margulis Family in order to obtain a neat formula for the unstable derivative (Lemma \ref{lem.formulalinda}), under the assumptions of Theorem~\ref{teo.maintechnical}. 
	
	Thus, along this section we consider $f:M\to M$ to be a partially hyperbolic diffeomorphism. Let $\mu\in\cp$ be an ergodic $u$-Gibbs measure admiting an unstable Margulis family. Thus, we are precisely under the assumptions of Theorem~\ref{teo.maintechnical}.

	\subsection{The Superposition Property Revisited}
	
	We denote by $\{m^u_x\}_{x\in\TT}$ the Margulis system along unstable manifolds. We denote by $\{\mu^u_x\}_{x\in\TT}$ the leafwise unstable measures of $\mu$, as constructed in the previous section.

	\begin{lemma}
		\label{lem.seba}
		Given a subordinate partition $\xi<W^u$ it holds for $\mu$-a.e. $x\in M$ that
		\[
		\frac{\mu^u_x|_\xi}{\mu^u_x(\xi(x))}=\frac{m^u_x|_\xi}{m^u_x(\xi(x))}
		\]
	\end{lemma}
	\begin{proof}
		Let $B$ be a foliated box in $M$ for the unstable foliation $\cW^u$ and denote by $\{\mu_x^B\}_{x\in B}$ be the disintegration of $\mu|_B$ along unstable plaques in $B$. Also, let $\{\mu_x^\xi\}_{x\in\TT}$ be the disintegration of $\mu$ with respect to the subordinate partition $\xi$.
		
		By Theorem~\ref{teo:leafwise} we know that
		\[\tag{i}
		\frac{\mu^u_x|_\xi}{\mu^u_x(\xi(x))}=\mu_x^\xi.
		\]
		Also, since $\{m^u_x\}_{x\in M}$ is a Margulis family of unstable measures, we know that there is a $C>0$ such that
		\[\tag{ii}
		m^u_x|_{B\cap W^u_{loc}(x)}=C\mu_x^B
		\]
		To connect these expressions, we must understand how $\mu_x^B$ and $\mu_x^\xi$ are related. For this, let us recall that almost every atom of a subordinate partition contains an open segment of unstable leaf (item (i) of definition \ref{def:subordinated}); thus for $\mu$-a.e.$x\in\TT$ the unstable leaf of $x$ contains atmost countably many atoms of $\xi$. In particular, there is a set $B'\subseteq B$ with $\mu(B\setminus B')=0$ such that for all $x\in B'$, the collection $\{\xi(y)\}_{y\in W^u_{loc}(x)\cap B}$ is a countable partition of $W^u_{loc}(x)\cap B$. Thus, we can disintegrate $\mu_x^B$ with respect to it; moreover, since the partition is countable, the conditionals of this disintegration are just the normalized restriction of $\mu_x^B$: for every $\phi\in L^1(B',m_x^B)$,
		\[
		\int\varphi d\mu_x^B=\int\left[\frac{\int\varphi(z) d\mu_x^B|_{\xi(y)}(z)}{\mu_x^B(\xi(y))}\right]d\mu_x^B|_{B'}(y)
		\]
		Thus, for all $\varphi\in L^1(\TT,\mu)$
		\[\begingroup
		\renewcommand{\arraystretch}{3}
		\begin{array}{rl}
			\displaystyle\int_B\varphi(x) d\mu(x)&\displaystyle=\int_{B'}\left[\int\left[\frac{\int\varphi(z) d\mu_x^B|_{\xi(y)}(z)}{\mu_x^B(\xi(y))}\right]d\mu_x^B|_{B'}(y)\right]d\mu(x)\\
		\end{array}
		\endgroup\]
		Using that $\{\mu_x^{B}\}_{x\in B'}$ is a disintegration of $\mu$ in $B'$, we can use Rokhlin's theorem to the function
		\[
		\psi(y)\eqdef\frac{\int\varphi(z)d\mu_x^B|_{\xi(y)}(z)}{\mu_x^B(\xi(y))}
		\]
		to obtain
		\[
		\int_B\varphi(x) d\mu(x)=\int_{B'}\left[\frac{\int\varphi(z)d\mu_x^B|_{\xi(x)}(z)}{\mu_x^B(\xi(x))}\right]d\mu(x).
		\]
		We deduce that $\{\frac{\mu^B_x|_\xi}{\mu^B_x(\xi(x))}\}_{x\in B'}$ is a system of conditional measures disintegrating $\mu$ with respect to $\xi$ restricted to $B'$. By the uniqueness in Rokhlin's Theorem we conclude that 
		\[
		\frac{\mu^B_x|_\xi}{\mu^B_x(\xi(x))}=\mu_x^\xi
		\]
		This, together with equations (i) and (ii) gives
		\[
		\frac{m^u_x|_{\xi(x)\cap W^u(x)\cap B'}}{m^u_x(\xi(x))}=\frac{\mu^u_x|_\xi}{\mu^u_x(\xi(x))}
		\]
		Since $B'$ has full measure in $B$, the above equation holds for $B$ instead of $B'$. Also the intersection with $W^u(x)$ is redundant, because $\xi(x)\subseteq W^u$. And, at last, we can cover $M$ with foliated boxes $B$ to obtain the equality almost everywhere as desired.
	\end{proof}
	
	\subsection{leafwise Jacobian formula under the existence of a Margulis system}
	
	The main finding of our work is the formula below giving a very precise estimate for the unstable derivative at a $\mu$-typical point. We are going to use along this section the following simplifying notation for orbits: given $x\in M$, and given an integer $k\in\Z$, we denote $x_k\eqdef f^k(x)$.
	
	\begin{lemma}
		\label{lem.formulalinda}
		Given any subordinate partition $\xi<\cW^u$, for almost every point and every integer $k\in\Z$ the following holds:
		\[
		\jac^u_x(k)^{-1}=\frac{g_{x_k}(x_k)}{g_x(x)}\times\frac{\mu^u_x(\xi(x))}{m^u_x(\xi(x))}\times m^u_{x_k}(\cW^u_1(x_k))\times e^{-\lambda k}.
		\]
	\end{lemma}
	\begin{proof}
		Let $x$ be a $\mu$ typical point and consider $A_k\eqdef\cW^u_1(x_k)$. Notice that, by our choice of normalization for leafwise measures, we have
		\[
		\mu^u_{x_k}(A_k)=1.
		\]
		Thus, applying Corollary~\ref{cor.leafwisegibbs} we deduce 
		\begin{align*}
1=\frac{g_{x_k}(x_k)}{g_x(x)}\times\jac^u_x(k)\times\mu^u_x(f^{-k}(A_k)).
		\end{align*}
		Hence,
		\begin{align*}
		\frac{\jac^u_x(k)^{-1}}{\mu^u_x(\xi(x))}\times\frac{g_x(x)}{g_{x_k}(x_k)}=\frac{\mu^u_x\circ f^{-k}(A_k)}{\mu^u_x(\xi(x))}=\frac{m^u_x\circ f^{-k}(A_k)}{m^u_x(\xi(x))}=\frac{e^{-\lambda k}m^u_{x_k}(A_k)}{m^u_x(\xi(x))}
		\end{align*}
		concluding the lemma.
	\end{proof}
	
	\subsection{Uniform bound on the dynamical density}
	
	Recall the dynamical density function $\rho^u_x(y)=\lim_{n\to-\infty}\frac{\jac^u_x(-n)}{\jac^u_y(-n)}$, for $y\in\cW^u(x)$, we defined before. Usually this function captures the accumulation of non-linearities of the map and is unbounded on every unstable leaf. Here, we show that the assumption that the Gibbs $u$-state $\mu$ admits an unstable Margulis family leads to the highly restrictive scenario when $\rho^u_x$ is bounded uniformly on entire leafs for $\mu$ almost every point. Given positive real numbers $a,b,C$ we say that $a\asymp_C b$ if $C^{-1}a\leq b\leq Ca$.
	
	\begin{proposition}
		\label{prop.main}
		There exists some $C>1$ and a full measure set $\Lambda\subset M$ such that whenever $x\in\Lambda$ it holds
		\[
		\rho^u_x(y)\asymp_C1,
		\]  
		for every $y\in\cW^u(x)$
	\end{proposition}
	
	In the proof of this proposition we are going to use a density argument. Recall that an integer subset $\cN\subset\N$ is said to have (Banach) density $\gamma\geq 0$ if 
	\[
	\lim_{n\to+\infty}\frac{\card\cN\cap[0,n-1]}{n}=\gamma. 
	\]
	Observe that the existence of the limit is part of the definition. One may, of course, speak about upper and lower density by using upper and lower limits, but this won't be necessary for our use of this concept, which is contained in the following simple lemma.
	
	\begin{lemma}
		\label{lem.birkhoff}
		Let $(g,X,\cB,\nu)$ be an ergodic measure preserving dynamical system acting on a probability space $(X,\cB,\nu)$. Let $K\in\cB$ be a measurable set with positive measure. Then, for $\nu$ almost every point $x\in X$ the set 
		\[
		\cN(x)\eqdef\{\ell\in\N;g^{\ell}(x)\in K\}
		\]
		has Banach density equal to $\nu(K)$. 
	\end{lemma} 
	\begin{proof}
		By ergodicity and by Birkhoff's ergodic theorem applied to the indicator function $\chi_K$ of $K$, we have that for $\nu$ almost every $x\in X$ it holds
		\[
		\lim_{n\to+\infty}\frac{\card\{\ell\in\N;g^{\ell}(x)\in K, l<n\}}{n}=\nu(K),
		\]	
		which is exactly the same as to say that $\cN(x)$ has Banach density equal to $\nu(K)$.
	\end{proof}
	
	We can now complete the proof of Proposition~\ref{prop.main}
	
	\begin{proof}[Proof of Proposition~\ref{prop.main}]
		We know from Lemma~\ref{lem.formulalinda} that for $\mu$ a.e. $x\in M$ it holds
		\begin{equation}
			\label{eq.linda}
			\jac^u_x(k)^{-1}=\frac{g_{x_k}(x_k)}{g_x(x)}\times\frac{\mu^u_x(\xi(x))}{m^u_x(\xi(x))}\times m^u_{x_k}(\cW^u_1(x_k))\times e^{-\lambda k}.
		\end{equation}
		Chaging variables in \eqref{eq.linda} we get
		\begin{equation}
			\label{eq.lindadois}
			\jac^u_x(-k)=\frac{g_{x}(x)}{g_{x_{-k}}(x_{-k})}\times\frac{\mu^u_{x_{-k}}(\xi(x_{-k}))}{m^u_{x_{-k}}(\xi(x_{-k}))}\times m^u_x(\cW^u_1(x))\times e^{-\lambda k}.
		\end{equation}
		Apply Lusin's theorem to the measurable function $x\in M\mapsto \frac{\mu^u_x(\xi(x))}{m^u_x(\xi(x))}\times g_x(x)$ to find a compact set $K$ with $\mu(K)>0.5$ where this function is uniformly bounded, from above and below. Applying Lemma~\ref{lem.birkhoff} for this $K$ with $g=f^{-1}$ we deduce that for $\mu$-almost every $x\in M$ there exists a subset $\cN(x)\subset\N$ with density greater than 0.5 such that $k\in\cN(x)$ if, and only if, $f^{-k}(x)=x_{-k}\in K$. 
		
		By the very definition of $K$ and \eqref{eq.lindadois} this shows that for $\mu$-a.e. $x\in M$ it holds
		\[
		\frac{\jac^u_x(-k)}{e^{-k\lambda}}\asymp_C 1,
		\]
		for some uniform constant $C\geq 1$ and $k\in\cN(x)\subset\N$. Denote by $\Gamma$ the set of $x\in M$ with this property, for a fixed constant $C$. Denote by $\xi_n=f^n(\xi)$ and consider
		\[
		\Gamma_n\eqdef\{x\in\Gamma;\mu^{\xi_n}_x(\Gamma\cap\xi_n(x))=1\}.
		\]
		By Rokhlin's theorem it follows that $\mu(\Gamma_n)=1$, for every $n\in\N$. We define $\Lambda=\cap_{n\in\N}\Gamma_n$, which is a full measure set. Take a point $x\in\Lambda$ and take any point $y\in\cW^u(x)$. Take $n\in\N$ large enough so that $\cW^u_{\varepsilon}(y)\subset\xi_n(x)$, for some $\varepsilon>0$. Since $\mu^{\xi_n}_x<<\operatorname{Leb}^u_x$ there must exist some $z\in\cW_{\varepsilon}^u(y)\cap\Lambda$. Thus, we have
		\[
		\frac{\jac^u_x(-k)}{\jac^u_y(-k)}=\frac{\jac^u_x(-k)}{e^{-k\lambda}}\times\frac{e^{-k\lambda}}{\jac^u_z(-k)}\times\frac{\jac^u_z(-k)}{\jac^u_y(-k)}\asymp_{C^3}1,
		\]
		for every $k$ in the subset $\cN(x)\cap\cN(z)$ which has positive density. Let $k_n$ be an enumeration of this set. The result now follows because
		\[
		\lim_{n\to\infty}\frac{\jac^u_x(-k_n)}{\jac^u_y(-k_n)}=\rho^u_x(y)\qedhere
		\]
	\end{proof}	
	
\section{Bounded density implies constant periodic data}\label{sec.final}

In this section we shall complete the proof of Theorem~\ref{main.technical} combining the results of the previous section with the following.

\begin{theorem}
	\label{teo.densiteborne}
Let $f:M\to M$ be a partially hyperbolic diffeomorphism with minimal unstable foliation. Assume that there exists some $u$-Gibbs measure $\mu\in\cp$ so that the family of densities $\{g_x\}_{x\in M}$, given by Proposition~\ref{prop.leafwisegibbs}, satisfy: there exists a measurable set $\Lambda\subset M$ with $\mu(\Lambda)>0$ so that for every $x\in\Lambda$ it holds
\[
0<\inf_{y\in\cW^u(x)}g_x(y)\leq\sup_{y\in\cW^u(x)}g_x(y)<\infty.
\]	
Then, $f$ has constant unstable Jacobian periodic data. 
	
\end{theorem}	
	
Let us show how to complete the proof of our main result using the above theorem.

\subsubsection{Proof of Theorem~\ref{main.technical}}
Assume that there exists some ergodic $u$-Gibbs measure $\mu\in\cp$ admiting an unstable Margulis family. By Proposition~\ref{prop.main} we obtain a uniform constant $C>1$ such that for almost every $x\in M$ and every $y\in\cW^u(x)$ it holds 
\[
\rho^u_x(y)\asymp_C1.
\] 	
By Proposition~\ref{prop.leafwisegibbs} we deduce that, for $\mu$-almost every $x\in M$ it holds
\[
C^{-1}g_x(x)\leq g_x(y)\leq g_x(x)C.
\] 
Applying Theorem~\ref{teo.densiteborne} we deduce
 $\jac^u(p)=\jac^u(q)$, for every $p,q\in\operatorname{Per}(f)$, concluding.\qed

\subsection{Proof of Theorem~\ref{teo.densiteborne}}

The rest of this section is devoted to the proof of Theorem~\ref{teo.densiteborne}. The key idea is to use the minimality of the strong unstable foliation in order to uniformize the  bound on the dynamical density. This is precisely the goal of the next lemma. 
	
	\begin{lemma}
		\label{lem.davi}
		Let $f:M\to M$ be a partially hyperbolic diffeomorphism with minimal unstable foliation. Assume that for some $x\in M$, it holds
		\[
		0<\inf_{y\in\cW^u(x)}\rho^u_x(y)\leq\sup_{y\in\cW^u(x)}\rho^u_x(y)<\infty.
		\]	
		Then, there exists a uniform constant $C>1$ such that for every $x\in M$ and every $y\in\cW^u(x)$ it holds
		\[
		\rho^u_x(y)\asymp_C 1.
		\]
	\end{lemma}
	\begin{proof}
 Let $C_0>1$ and $x\in M$ so that $\rho^u_{x}(y)\asymp_{C_0} 1$ for every $y\in\cW^u(x)$. We claim that the Lemma is true with $C\eqdef C_0^2$ as uniform bound. Assume by contradiction that this is not true. Then, for some $p\in M$, there exists some $y\in\cW^u(p)$ such that
		\[
		\rho^u_p(y)\notin[C_0^{-2},C_0^2].
		\]
By the minimality of the unstable foliation there exists points $\hat{x}$ as close as we want to $p$ and $\hat{y}$ as close as we want to $y$, both of them lying on $\cW^u(x)$. By continuity this implies that 
		\[
		\rho^u_{\hat{x}}(\hat{y})\notin[C_0^{-2},C_0^2].
		\]
		However, since
		\[
		\rho^u_{\hat{x}}(\hat{y})=\rho^u_{\hat{x}}(x)\rho^u_x(\hat{y})\in[C_0^{-2},C_0^2],
		\]
		by assumption, we achieve a contradiction.
	\end{proof}	
	
	As a general fact, we establish below that a uniform bound on the dynamical density function always implies constant unstable Jacobian periodic data. Combining Proposition~\ref{prop.leafwisegibbs} with Lemma~\ref{lem.davi}, Proposition~\ref{prop.paris} finishes the proof of Theorem~\ref{teo.densiteborne}.  
	
	\begin{proposition}
		\label{prop.paris}
		Let $f:M\to M$ be a partially hyperbolic diffeomorphism. Assume that there exists some $C>1$ such that for every $x\in M$ and every $y\in\cW^u(x)$ it holds 
		\[
		\rho^u_x(y)\asymp_C1,
		\]
		Then, $\jac^u(p)=\jac^u(q)$, for every $p,q\in\operatorname{Per}(f)$.
	\end{proposition}
	\begin{proof}
		Let $p,q\in\operatorname{Per}(f)$ be two arbitrary periodic points. Define, for the sake of this proof $k\eqdef\pi(p)\times\pi(q)$. Then,
		\[
		\jac^u_\star(nk)=\jac^u(\star)^{nk},
		\]
		for every $n\in\Z$ and $\star=p,q$. 
		
		\begin{figure}[h]
			\centering
			\begin{tikzpicture}[scale=0.8]
				\draw[red!60!black, thick] (-2,1) node[left]{$\cW^u(p)$}.. controls (.5,1) and (1,.5) .. (0,0) .. controls (-1,-.5) and (-.5,-1) .. (4,-1) ;
				\draw (-1.8,1) node{$\bullet$};
				\draw (-1.8,1) node[below]{$p$};
				\draw (3.5,-1) node{$\bullet$};
				\draw (3.5,-1) node[above]{$y_n$};
				\draw (3.5,-1.65) node{$\bullet$};
				\draw (3.5,-1.65) node[below]{$q$};
			\end{tikzpicture}
			\caption{\label{fig.figfinal} Apllying minmality to find $y_n\in\cW^u(p)$ which shadows $O(q)$ for a finite but long time.}
		\end{figure}
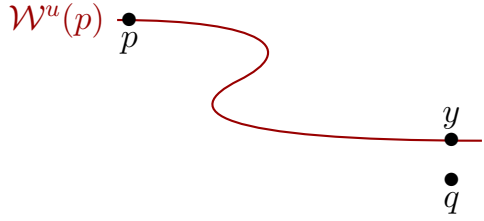
		
		By minimality of the unstable foliation, for each $n>>1$ there exists $y_n\in\cW^u(p)$ close enough to $q$ so that
		\[
		\frac{\jac^u_q(nk)}{\jac^u_{y_n}(nk)}\asymp_C 1.
		\]  
		This implies that 
		\begin{align*}
			\jac^u(p)^{-nk}&=\jac^u_p(-nk)\leq C\jac^u_{y_n}(-nk)\\
			&\leq C^2\jac^u_q(-nk)=C^2\jac^u(q)^{-nk}.
		\end{align*}
		Reversing the roles of $p$ and $q$ in this argument (using denseness of $\cW^u(q)$ instead) we get
		\[
		\jac^u(q)^{-nk}\leq C^2\jac^u(p)^{-nk},
		\] 
		for every $n>>1$. These inequalities imply directly that $\jac^u(p)=\jac^u(q)$. 
	\end{proof}

	\bibliographystyle{plain}
	\bibliography{Bib}
	
	\smallskip
	
	\begin{flushleft}
		
		{\scshape Vítor Gomes}\\
	Instituto de Matem\'atica e Estat\'istica, Universidade Federal Fluminense\\
	Rua Prof. Marcos Waldemar de Freitas Reis, S/N -- Blocos G e H\\
	Campus do Gragoatá, Niterói, Rio de Janeiro 24210-201, Brasil\\
	email:  \texttt{vitorgf@id.uff.br}	\\
		\vspace{1cm}	
		{\scshape Bruno Santiago}\\
	Instituto de Matem\'atica e Estat\'istica, Universidade Federal Fluminense\\
	Rua Prof. Marcos Waldemar de Freitas Reis, S/N -- Blocos G e H\\
	Campus do Gragoatá, Niterói, Rio de Janeiro 24210-201, Brasil\\
	email:  \texttt{brunosantiago@id.uff.br}

	\end{flushleft}
	
\end{document}